\documentclass[10pt,leqno]{article}

\usepackage{amsmath,amssymb,amsthm,mathrsfs,dsfont}

\usepackage[margin=3cm]{geometry} 

\usepackage{titlesec,hyperref}

\usepackage{color}

\usepackage{fancyhdr}
\titleformat{\subsection}{\it}{\thesubsection.\enspace}{1pt}{}

\newtheorem{theo}{Theorem}[section]
\newtheorem{lemm}[theo]{Lemma}

\newtheorem{prop}[theo]{Proposition}
\newtheorem{rema}[theo]{Remark}
\numberwithin{equation}{section}

\allowdisplaybreaks 

\begin{document}
\title{Optimal decay rate for the 2-D compressible Oldroyd-B and Hall-MHD model
\hspace{-4mm}
}

\author{ Zhaonan $\mbox{Luo}^1$ \footnote{Email: luozhn@fudan.edu.cn},\quad
Wei $\mbox{Luo}^2$\footnote{E-mail:  luowei23@mail2.sysu.edu.cn} \quad and\quad
Zhaoyang $\mbox{Yin}^{2,3}$\footnote{E-mail: mcsyzy@mail.sysu.edu.cn}\\
$^1\mbox{School}$ of Mathematical Sciences, Fudan University, Shanghai 200433, China.\\
$^2\mbox{Department}$ of Mathematics,
Sun Yat-sen University, Guangzhou 510275, China\\
$^3\mbox{Faculty}$ of Information Technology,\\ Macau University of Science and Technology, Macau, China}

\date{}
\maketitle
\hrule

\begin{abstract}
In this paper, we are concerned with long time behavior of the strong solutions to the 2-D compressible Oldroyd-B and Hall-MHD model. By virtue of the improved Fourier splitting method and the time weighted energy estimate, we obtain the $L^2$ decay rate $(1+t)^{-\frac{1}{4}}$. According to the Littlewood-Paley theory, we prove that the solutions belong to the critical Besov space with negative index. Finally, we show optimal decay rate in $H^2$-framework without the smallness restriction of low frequencies. \\
\vspace*{5pt}
\noindent {\it 2020 Mathematics Subject Classification}: 35Q30, 76A10, 76N10.

\vspace*{5pt}
\noindent{\it Keywords}: The compressible Oldroyd-B model; The compressible Hall-MHD equations; Besov spaces; Time decay rate.
\end{abstract}

\vspace*{10pt}

\tableofcontents

\section{Introduction}
In this paper we mainly study the compressible Oldroyd-B model \cite{1958Non,2018Global}:
\begin{align}\label{eq0}
\left\{
\begin{array}{ll}
\varrho_t+div(\varrho u)=0 , \\[1ex]
(\varrho u)_t+div(\varrho u\otimes u)-(1-\omega)(\Delta u+\nabla div~u)+\nabla_x{P(\varrho)}=div~\tau, \\[1ex]
\tau_t+u\cdot\nabla\tau+a\tau+g_b(\tau, \nabla u)=2\omega D(u),  \\[1ex]
\varrho|_{t=0}=\varrho_0,~~u|_{t=0}=u_0,~~\tau|_{t=0}=\tau_0. \\[1ex]
\end{array}
\right.
\end{align}
In \eqref{eq0}, $\varrho(t,x)$ stands for the density of the solvent, $u(t,x)$ denotes the velocity of the polymeric liquid and $\tau(t,x)$ represents the symmetric tensor of constrains. The parameters $\omega\in(0,1)$ stands for the coupling constant. The pressure satisfies $P(\varrho)=\varrho^\gamma$ with $\gamma\geq1$. The parameters satisfy $a\geq 0$ and $b\in[-1, 1]$. Moreover,
$$g_b(\tau, \nabla u)=\tau W(u)-W(u)\tau+b(D(u)\tau+\tau D(u)),$$
with the vorticity tensor $W(u)=\frac {\nabla u-(\nabla u)^T} {2}$ and the deformation tensor $D(u)=\frac {\nabla u+(\nabla u)^T} {2}$.
For more explanations on the modeling, one can refer to $\cite{2014Incompressible}$ and $\cite{2018Global}$.

Then we introduce the compressible Hall-MHD equations \cite{Fan2015On,2017MHDGlobal}:
\begin{align}\label{M0}
\left\{
\begin{array}{ll}
\varrho_t+div(\varrho u)=0 , \\[1ex]
(\varrho u)_t+div(\varrho u\otimes u)-\mu\Delta u-(\mu+\lambda)\nabla div~u+\nabla_x{P(\varrho)}=(curlB)\times B, \\[1ex]
B_t-\nu\Delta B-curl(u\times B)+curl[\frac {(curlB)\times B}{\varrho}]=0, ~~~~ div~B=0, \\[1ex]
\varrho|_{t=0}=\varrho_0,~~u|_{t=0}=u_0,~~B|_{t=0}=B_0. \\[1ex]
\end{array}
\right.
\end{align}
In \eqref{M0}, $B(t,x)$ represents magnetic field. The constants $\mu$ and $\lambda$ denote the viscosity coefficients of the flow and satisfy
$\mu>0$ and $2\mu+3\lambda\geq 0$. The positive constant $\nu$ is the magnetic diffusivity acting as a magnetic diffusion coefficient of the magnetic field.
Moreover, $curl(\frac {(curlB)\times B}{\varrho})$ represents the Hall effect. Notice that
$$(curlB)\times B=(B\cdot\nabla)B-\frac 1 2 \nabla(|B|^2),$$
and
$$curl(u\times B)=u(div~B)-(u\cdot\nabla)B+(B\cdot\nabla)u-B(div~u).$$
For more explanations on the modeling, one can refer to $\cite{2017MHDGlobal}$.

Let $a=1$, $\omega=\frac 1 2$ and $x\in\mathbb{R}^2$. Notice that $(\varrho,u,\tau)=(1,0,0)$ is a trivial solution of \eqref{eq0}. Taking $\rho=\varrho-1$, $I(\rho)=\frac {\rho} {\rho+1}$ and $k(\rho)=\gamma I(\rho)+\frac {\gamma-P'(1+\rho)} {\rho+1}$, we can rewrite \eqref{eq0} as the following system:
\begin{align}\label{eq1}
\left\{
\begin{array}{ll}
\rho_t+div~u=F , \\[1ex]
u_t-\frac 1 2(\Delta+\nabla div)u+\gamma\nabla\rho-div~\tau=G, \\[1ex]
\tau_t+\tau-D(u)=H, \\[1ex]
\rho|_{t=0}=\rho_0,~~u|_{t=0}=u_0,~~\tau|_{t=0}=\tau_0, \\[1ex]
\end{array}
\right.
\end{align}
where $F=-div(\rho u)$, $G=-u\cdot\nabla u-\frac 1 2I(\rho)(\Delta+\nabla div)u-I(\rho) div~\tau+k(\rho)\nabla\rho$ and $H=-u\cdot\nabla\tau-g_b(\tau, \nabla u)$.

Let $\mu=\nu=1$ and $\lambda=0$. The equations \eqref{M0} can be rewritten as the following system:
\begin{align}\label{M1}
\left\{
\begin{array}{ll}
\rho_t+div~u=F_1 , \\[1ex]
u_t-(\Delta+\nabla div)u+\gamma\nabla\rho=G_1, \\[1ex]
B_t-\Delta B=H_1, \\[1ex]
\rho|_{t=0}=\rho_0,~~u|_{t=0}=u_0,~~B|_{t=0}=B_0, \\[1ex]
\end{array}
\right.
\end{align}
where $F_1=-div(\rho u)$, $G_1=-u\cdot\nabla u-I(\rho)(\Delta+\nabla div)u-\frac {(curlB)\times B} {1+\rho}+k(\rho)\nabla\rho$
and $H_1=curl(u\times B)-curl[\frac {(curlB)\times B}{\rho+1}]$.

\subsection{The incompressible Oldroyd-B model}
In \cite{Guillope1990}, C. Guillop\'e,  and J. C. Saut first showed that the incompressible Oldroyd-B model admits a unique global strong solution in Sobolev spaces. The $L^p$-setting was given by E. Fern\'andez-Cara, F.Guill\'en and R. Ortega \cite{Fernandez-Cara}. The weak solutions of the incompressible Oldroyd-B model was proved by P. L. Lions and N. Masmoudi \cite{Lions-Masmoudi} for the case $\alpha=0$. Notice that the problem for the case $\alpha\neq0$ is still open, see \cite{2011Global,Masmoudi2013}. Later on, J. Y. Chemin and N. Masmoudi \cite{Chemin2001}
proved the existence and uniqueness of strong solutions in homogenous Besov spaces with critical index of regularity. Optimal decay rates for solutions to the 3-D incompressible Oldroyd-B model were obtained by M. Hieber, H. Wen and R. Zi \cite{2019OldroydB}. The sharp time decay rates of large solutions to the two-dimensional Oldroyd-B model were proved in \cite{Li2}. An approach based on the deformation tensor can be found in \cite{Li1,Lei-Zhou2005,Lei2008,2010On,2010On1,Zhang-Fang2012}.

\subsection{The compressible Oldroyd-B model}
Z. Lei \cite{2006Global} first investigated the incompressible limit problem of the compressible Oldroyd-B model in a torus. Recently, D. Fang and R. Zi \cite{2014Incompressible} studied the global well-posedness for compressible Oldroyd-B model in critical Besov spaces with $d\geq 2$. In \cite{2018Global}, Z. Zhou, C. Zhu and R. Zi proved the global well-posedness and decay rates for the 3-D compressible Oldroyd-B model. For the compressible Oldroyd-B type model based on the deformation tensor can be found in \cite{2010Global,2017viscoelastic,Suli2016}.

\subsection{The compressible Hall-MHD equations}
In \cite{Fan2015On}, J. S. Fan, A. Alsaedi, T. Hayat, G. Nakamura and Y. Zhou established global existence and optimal decay rates of solutions for 3-D compressible Hall-MHD equations in $H^3$-framework. Recently, $H^2$-framework for the same problem was studied by Z. A. Yao and J. Gao \cite{2017MHDGlobal}.

The compressible Hall-MHD equations reduce to the compressible MHD equations when the Hall effect term $curl(\frac {(curlB)\times B}{\varrho})$ is neglected. Then we cite some reference about the compressible MHD equations. For $d=2$, S. Kawashima \cite{1984Smooth} obtained the global existence of smooth solution to the general electromagnetic fluid equations. Recently, X. Hu and D. Wang \cite{2010MHDGlobal,2008MHDGlobal} established the existence and large-time behavior of global weak solutions with large data in a bounded domain. A. Suen and D. Hoff \cite{2012MHDGlobal} obtained the global low-energy weak solutions where initial data are chosen to be small and initial densities are assumed to be nonnegative and essentially bounded. F. C. Li and H. J. Yu \cite{2011MHDOptimal} and Q. Chen and Z. Tan \cite{2010MHDGlobal1} established the global existence of solution and obtained the decay rate of solution for the 3-D compressible MHD equations. Large time behavior of strong solutions to the compressible MHD system in the critical $L^p$ framework with $d\geq 2$ was proved by W. Shi and J. Xu in \cite{2018MHDLarge}. In \cite{2019MHDOptimal}, Q. Bie, Q. Wang and Z. A. Yao proved optimal decay for the compressible MHD equations in the critical regularity framework and and removed the smallness assumption of low frequencies.

\subsection{Short review for the CNS equations}
Taking $B=0$, the system \eqref{M0} reduce to the well-known compressible Navier-Stokes (CNS) equations. In order to study about the large time behaviour for the \eqref{eq0}, we cite some reference about the CNS equations. The large time behaviour of the global solutions $(\rho,u)$ to the 3-D CNS equations was firstly proved by A. Matsumura and T. Nishida in \cite{Matsumura}. Recently, H. Li and T. Zhang \cite{Li2011Large} obtained the optimal time decay rate for the 3-D CNS equations by spectrum analysis in Sobolev spaces. R. Danchin and J. Xu \cite{2016Optimal} studied about the large time behaviour in the critical Besov space with $d\geq 2$. J. Xu \cite{Xu2019} obtained the optimal time decay rate with a small low-frequency assumption in some Besov spaces with negative index. More recently, Z. Xin and J. Xu \cite{2018Optimal} studied about the large time behaviour and removed the smallness assumption of low frequencies.

\subsection{Main results}
The long time behavior for polymeric models is noticed by N. Masmoudi \cite{2016Equations}. To our best knowledge, large time behaviour for the 2-D compressible Oldroyd-B system \eqref{eq1} has not been studied yet. This problem is interesting and more difficult than the case with $d\geq 3$. In this paper, we firstly study about optimal time decay rate for the 2-D compressible Oldroyd-B system in $H^2$-framework. The proof is based on the Littlewood-Paley decomposition theory and the improved Fourier splitting method. Similar to \cite{He2009} and \cite{Luo-Yin}, we first cancel the linear term in Fourier space. By the Fourier splitting method and the bootstrap argument, we firstly obtain initial logarithmic decay rate 
\begin{align*}
\|(\rho,u,\tau)\|_{L^2}\leq C\ln^{-l}(e+t),
\end{align*}
for any $l\in N^{+}$. The main difficulty for us is to get the initial polynomial decay rate. For lack of time information for global solutions in $L^1$ from \eqref{eq1}, we otain 
\begin{align*}
\int_{S(t)}\int_{0}^{t}|\mathcal{F}(k(\rho)\nabla\rho)\cdot\bar{\hat{u}}|ds'd\xi\leq C(1+t)^{-\frac 1 2} \int_{0}^{t}\|\rho\|_{L^{2}}\|\nabla \rho\|_{L^{2}}\|u\|_{L^{2}}ds'.
\end{align*}
By virtue of the time weighted energy estimate and logarithmic decay rate, we improve the time decay rate to
\begin{align*}
\|(\rho,u,\tau)\|_{L^2}\leq C(1+t)^{-\frac{1}{4}}.
\end{align*} 
Notice that the time decay rate we obtained is not the optimal decay rate. However, we can prove $(\rho,u,\tau)\in L^\infty(0,\infty; \dot{B}^{-\frac 1 2}_{2,\infty})$ from \eqref{eq1} by using the time decay rate $(1+t)^{-\frac{1}{4}}$. Then we improve the time decay rate to $(1+t)^{-\frac{5}{16}}$ by the Littlewood-Paley decomposition theory and the Fourier splitting method. We deduce a slightly weaker conclusion $$(\rho,u,\tau)\in L^\infty(0,\infty; \dot{B}^{-1}_{2,\infty})$$ from \eqref{eq1} by using the time decay rate $(1+t)^{-\frac{5}{16}}$. Without the smallness restriction of low frequencies, we obtain optimal time decay rate 
\begin{align*}
\|(\rho,u)\|_{L^2}\leq C(1+t)^{-\frac 1 2}
\end{align*}
by the Littlewood-Paley decomposition theory and the Fourier splitting method. Moreover, we can prove the faster time decay rate for $\tau$ with $$\|\tau\|_{L^2}\leq C(1+t)^{-1}.$$ Finally, we apply the methods to the 2-D compressible Hall-MHD model \eqref{M1}. Processing Hall effect term through magnetic diffusion effect, we prove optimal time decay without small low-frequency assumption in critical Besov spaces with negative index. Notice that improved Fourier splitting method can be widely used
for other complex systems in the future.

Our main result can be stated as follows.
\begin{theo}\label{th2}
Let $d=2$. Let $(\rho,u,\tau)$ be a global strong solution of \eqref{eq1} with the initial data $(\rho_0,u_0,\tau_0)$ under the condition in Theorem \ref{th1}. In addition, if $(\rho_0,u_0,\tau_0)\in \dot{B}^{-1}_{2,\infty}$, then there exists a constant $C$ such that
\begin{align}
\|(\rho,u)\|_{L^2}\leq C(1+t)^{-\frac 1 2},~~~\|\tau\|_{L^2}\leq C(1+t)^{-1},
\end{align}
and
\begin{align}
\|\nabla(\rho,u,\tau)\|_{H^1}\leq C(1+t)^{-1}.
\end{align}
\end{theo}

\begin{rema}
Combining with the result in \cite{Xu2019,Li2}, one can see that the time decay rate for $(\rho,u)$ obtained in Theorem \ref{th2} is optimal.
\end{rema}

\begin{rema}
In previous papers, researchers usually add the condition $(\rho_0,u_0,\tau_0)\in L^1$ to obtain the optimal time decay rate. Since $L^1\hookrightarrow \dot{B}^{-1}_{2,\infty}$, it follows that our condition is weaker and the results still hold true for $(\rho_0, u_0,\tau_0)\in L^1$. Moreover, the assumption can be replaced with a weaker assumption $\sup_{j\leq j_0}2^{-j}\|\dot{\Delta}_j (\rho_0,u_0,\tau_0)\|_{L^2} <\infty$, for any $j_0\in \mathbb{Z}$.
\end{rema}

\begin{theo}\label{th4}
Let $d=2$. Let $(\rho,u,B)$ be a global strong solution of \eqref{M1} with the initial data $(\rho_0,u_0,B_0)$ under the condition in Theorem \ref{th3}. In addition, if $(\rho_0,u_0,B_0)\in \dot{B}^{-1}_{2,\infty}$, then there exists a constant $C$ such that
\begin{align}
\|(\rho,u,B)\|_{L^2}\leq C(1+t)^{-\frac 1 2},
\end{align}
and
\begin{align}
\|\nabla(\rho,u,B)\|_{H^1}\leq C(1+t)^{-1}.
\end{align}
\end{theo}
\begin{rema}
	Combining with the result in \cite{Xu2019}, one can prove that the $H^1$ decay rate for $(\rho,u,B)$ obtained in Theorem \ref{th4} is optimal.
\end{rema}

The paper is organized as follows. In Section 2 we introduce some notations and  give some preliminaries which will be used in the sequel. In Section 3 we prove the time decay rate of solutions to the 2-D compressible Oldroyd-B model
by using the Fourier splitting method, the Littlewood-Paley decomposition theory and the bootstrap argument. In Section 4 we apply the methods to proving the time decay rate for the 2-D compressible Hall-MHD model.


\section{Preliminaries}
In this section we introduce some notations and useful lemmas which will be used in the sequel.

The symbol $\widehat{f}=\mathcal{F}(f)$ represents the Fourier transform of $f$.
Let $\Lambda^s f=\mathcal{F}^{-1}(|\xi|^s \widehat{f})$.
We agree that $\nabla$ stands for $\nabla_x$ and $div$ stands for $div_x$.

We now recall the Littlewood-Paley decomposition theory and and Besov spaces.
\begin{prop}\cite{Bahouri2011}\label{pro0}
Let $\mathcal{C}$ be the annulus $\{\xi\in\mathbb{R}^2:\frac 3 4\leq|\xi|\leq\frac 8 3\}$. There exist radial function $\varphi$, valued in the interval $[0,1]$, belonging respectively to $\mathcal{D}(\mathcal{C})$, and such that
$$ \forall\xi\in\mathbb{R}^2\backslash\{0\},\ \sum_{j\in\mathbb{Z}}\varphi(2^{-j}\xi)=1, $$
$$ |j-j'|\geq 2\Rightarrow\mathrm{Supp}\ \varphi(2^{-j}\cdot)\cap \mathrm{Supp}\ \varphi(2^{-j'}\cdot)=\emptyset. $$
Further, we have
$$ \forall\xi\in\mathbb{R}^2\backslash\{0\},\ \frac 1 2\leq\sum_{j\in\mathbb{Z}}\varphi^2(2^{-j}\xi)\leq 1. $$
\end{prop}

Let $u$ be a tempered distribution in $\mathcal{S}'_h(\mathbb{R}^2)$. For all $j\in\mathbb{Z}$, define
$$\dot{\Delta}_j u=\mathcal{F}^{-1}(\varphi(2^{-j}\cdot)\mathcal{F}u).$$
Then the Littlewood-Paley decomposition is given as follows:
$$ u=\sum_{j\in\mathbb{Z}}\dot{\Delta}_j u \quad \text{in}\ \mathcal{S}'(\mathbb{R}^2). $$
Let $s\in\mathbb{R},\ 1\leq p,r\leq\infty.$ The homogeneous Besov space $\dot{B}^s_{p,r}$ is defined by
$$ \dot{B}^s_{p,r}=\{u\in \mathcal{S}'_h:\|u\|_{\dot{B}^s_{p,r}}=\Big\|(2^{js}\|\dot{\Delta}_j u\|_{L^p})_j \Big\|_{l^r(\mathbb{Z})}<\infty\}.$$

The following lemma describes inclusions between Lesbesgue and Besov spaces
\begin{lemm}\cite{Bahouri2011}\label{Lemma}
Let $1\leq p\leq 2$ and $d=2$. Then we have
$$L^{p}\hookrightarrow \dot{B}^{1-\frac 2 p}_{2,\infty}.$$
\end{lemm}

The following lemma is the Gagliardo-Nirenberg inequality of Sobolev type.
\begin{lemm}\cite{1959On}\label{Lemma0}
Let $d=2,~p\in[2,+\infty)$ and $0\leq s,s_1\leq s_2$, then there exists a constant $C$ such that
 $$\|\Lambda^{s}f\|_{L^{p}}\leq C \|\Lambda^{s_1}f\|^{1-\theta}_{L^{2}}\|\Lambda^{s_2} f\|^{\theta}_{L^{2}},$$
where $0\leq\theta\leq1$ and $\theta$ satisfy
$$ s+2(\frac 1 2 -\frac 1 p)=s_1 (1-\theta)+\theta s_2.$$
Note that we require that $0<\theta<1$, $0\leq s_1\leq s$, when $p=\infty$.
\end{lemm}

\section{Optimal time decay rate for the compressible Oldroyd-B model}
This section is devoted to investigating the long time behaviour for the 2-D compressible Oldroyd-B model.
We first introduce the energy and energy dissipation functionals for $(\rho,u,\tau)$ as follows:
$$E_\sigma=\|\Lambda^\sigma(\rho,u,\tau)\|^2_{H^{2-\sigma}}+2\eta\langle \Lambda^\sigma u, \nabla \Lambda^\sigma\rho \rangle_{H^{1-\sigma}},$$
and
$$D_\sigma=\eta\gamma\|\nabla\Lambda^\sigma\rho\|^2_{H^{1-\sigma}}+\frac 1 2\|\nabla \Lambda^\sigma u\|^2_{H^{2-\sigma}}+\frac 1 2\|div~\Lambda^\sigma u\|^2_{H^{2-\sigma}}+\|\Lambda^\sigma\tau\|^2_{H^{2-\sigma}},$$
where $\sigma=0$ or 1 and $\eta$ is a sufficiently small constant.

Using the energy
method and interpolation theory in \cite{2018Global}, one can easily deduce that the global existence of strong solutions for \eqref{eq1}. Thus we omit the
proof here and present the following Theorem.
\begin{theo}\label{th1}
Let $d=2$. Let $(\rho,u,\tau)$ be a strong solution of \eqref{eq1} with initial data $(\rho_0,u_0,\tau_0)\in H^2$. Then, there exists some sufficiently small constant $\epsilon_0$ such that if
\begin{align*}
E_0(0)\leq \epsilon_0,
\end{align*}
then \eqref{eq1} admits a unique global strong solution $(\rho,u,\tau)$. Moreover, for $\sigma=0$ or 1 and any $t>0$, we have
\begin{align}\label{ineq0}
\frac {d} {dt} E_\sigma(t)+D_\sigma(t)\leq 0.
\end{align}
\end{theo}

Since the additional stress tensor $\tau$ does not decay fast enough, we failed to use the bootstrap argument as in \cite{Schonbek1985,Luo-Yin2}. Similar to \cite{He2009} and \cite{2018Global}, we consider the coupling effect between $(\rho,u ,\tau)$. By taking Fourier transform in \eqref{eq1} and using the Fourier splitting method, we obtain the initial $L^2$ decay rate in following Proposition.
\begin{prop}\label{prop1}
Let $(\rho_0,u_0,\tau_0)\in \dot{B}^{-1}_{2,\infty}$. Under the condition in Theorem \ref{th1},
for any $l\in N^{+}$, then there exists a constant $C$ such that
\begin{align}\label{decay1}
E_0(t)\leq C\ln^{-l}(e+t),~~~~E_1(t)\leq C(1+t)^{-1}\ln^{-l}(e+t).
\end{align}
\end{prop}
\begin{proof}
Taking $\sigma=0$ in \eqref{ineq0}, we first have the following global energy estimation:
\begin{align}\label{ineq1}
\frac d {dt} E_0(t)+D_0(t)\leq 0.
\end{align}
Define $S_0(t)=\{\xi:|\xi|^2\leq 2C_2\frac {f'(t)} {f(t)}\}$ with $f(t)=\ln^{3}(e+t)$ and $C_2$ large enough. Applying Schonbek's \cite{Schonbek} strategy to \eqref{ineq1}, we obtain
\begin{align}\label{ineq2}
&\frac d {dt} [f(t)E_0(t)]+C_2 f'(t)(\frac 1 2\|u\|^2_{H^{2}}+\eta\gamma\|\rho\|^2_{H^{1}})
+f(t)\|\tau\|^2_{H^{2}}  \\ \notag
&\leq Cf'(t)\int_{S_0(t)}|\hat{u}(\xi)|^2+|\hat{\rho}(\xi)|^2d\xi+2f'(t)\|\nabla^2 \rho\|^2_{L^2}.
\end{align}

The $L^2$ estimate to the low frequency part of $(\rho,u)$ play a key role in studying time decay rates. Applying Fourier transform to \eqref{eq1}, we obtain
\begin{align}\label{eq2}
\left\{
\begin{array}{ll}
\hat{\rho}_t+i\xi_{k} \hat{u}^k=\hat{F},  \\[1ex]
\hat{u}^{j}_t+\frac 1 2|\xi|^2 \hat{u}^j+\frac 1 2 \xi_{j} \xi_{k} \hat{u}^k+i\xi_{j} \gamma\hat{\rho}-i\xi_{k} \hat{\tau}^{jk}=\hat{G}^j,  \\[1ex]
\hat{\tau}^{jk}_t+\hat{\tau}^{jk}-\frac i 2(\xi_{k} \hat{u}^j+\xi_{j} \hat{u}^k)=\hat{H}^{jk}.\\[1ex]
\end{array}
\right.
\end{align}
One can verify that
$$\mathcal{R}e[i\xi\cdot\hat{u}\bar{\hat{\rho}}]+\mathcal{R}e[\hat{\rho}i\xi\cdot\bar{\hat{u}}]
=\mathcal{R}e[i\xi\otimes\bar{\hat{u}}(t,\xi):\hat{\tau}]+\mathcal{R}e[i\xi\otimes\hat{u}:\bar{\hat{\tau}}]=0.$$
Since $\tau$ is symmetric, using \eqref{eq2}, we deduce that
\begin{align}\label{eq3}
&\frac 1 2 \frac d {dt} (\gamma|\hat{\rho}|^2+|\hat{u}|^2+|\hat{\tau}|^2)+\frac 1 2|\xi|^2 |\hat{u}|^2+\frac 1 2 |\xi\cdot\hat{u}|^2
+|\hat{\tau}|^2  \\ \notag
&=\mathcal{R}e[\gamma\hat{F}\bar{\hat{\rho}}]+\mathcal{R}e[\hat{G}\cdot\bar{\hat{u}}]+\mathcal{R}e[\hat{H}\bar{\hat{\tau}}].
\end{align}
Integrating \eqref{eq3} over $[0,t]$ with $s$, we obtain
\begin{align}\label{ineq3}
|\hat{\rho}|^2+|\hat{u}|^2+|\hat{\tau}|^2
\leq C(|\hat{\rho}_0|^2+|\hat{u}_0|^2+|\hat{\tau}_0|^2)+C\int_{0}^{t}|\hat{F}\cdot\bar{\hat{\rho}}|+|\hat{G}\cdot\bar{\hat{u}}|+|\hat{H}|^2 ds .
\end{align}
Integrating \eqref{ineq3} over $S_0(t)$ with $\xi$, then we have the following estimation to \eqref{eq2}:
\begin{align}\label{ineq4}
\int_{S_0(t)}|\hat{\rho}|^2+|\hat{u}|^2+|\hat{\tau}|^2 d\xi
&\leq C\int_{S_0(t)} |\hat{\rho}_0|^2+|\hat{u}_0|^2+|\hat{\tau}_0|^2 d\xi \\ \notag
&+C\int_{S_0(t)}\int_{0}^{t}|\hat{F}\cdot\bar{\hat{\rho}}|+|\hat{G}\cdot\bar{\hat{u}}|+|\hat{H}|^2 dsd\xi.
\end{align}
If $E_0(0)<\infty$ and $(\rho_0,u_0,\tau_0)\in \dot{B}^{-1}_{2,\infty}$, applying Proposition \ref{pro0}, we have
\begin{align}\label{ineq5}
\int_{S_0(t)}(|\hat{\rho}_0|^2+|\hat{u}_0|^2+|\hat{\tau}_0|^2)d\xi
&\leq\sum_{j\leq \log_2[\frac {4} {3}C_2^{\frac 1 2 }\sqrt{\frac {f'(t)}{f(t)}}]}\int_{\mathbb{R}^{2}} 2\varphi^2(2^{-j}\xi)(|\hat{\rho}_0|^2+|\hat{u}_0|^2+|\hat{\tau}_0|^2)d\xi  \\ \notag
&\leq\sum_{j\leq \log_2[\frac {4} {3}C_2^{\frac 1 2 }\sqrt{\frac {f'(t)}{f(t)}}]}C(\|\dot{\Delta}_j u_0\|^2_{L^2}+\|\dot{\Delta}_j \rho_0\|^2_{L^2}+\|\dot{\Delta}_j \tau_0\|^2_{L^2}) \\ \notag
&\leq C\|(\rho_0,u_0,\tau_0)\|^2_{\dot{B}^{-1}_{2,\infty}}\sum_{j\leq\log_2[\frac {4} {3}C_2^{\frac 1 2 }\sqrt{\frac {f'(t)}{f(t)}}]}2^{2j} \\ \notag
&\leq C\frac {f'(t)}{f(t)}\|(\rho_0,u_0,\tau_0)\|^2_{\dot{B}^{-1}_{2,\infty}}.
\end{align}
By Minkowski's inequality and \eqref{ineq0}, we get
\begin{align}\label{ineq6}
\int_{S_0(t)}\int_{0}^{t}|\hat{H}|^2 dsd\xi&=\int_{0}^{t}\int_{S_0(t)}|\hat{H}|^2 d\xi ds  \\ \notag
&\leq C\frac {f'(t)}{f(t)} \int_{0}^{t}\|u\|^2_{L^{2}}\|\nabla \tau\|^2_{L^{2}}+\|\nabla u\|^2_{L^{2}}\|\tau\|^2_{L^{2}}ds \\ \notag
&\leq C\frac {f'(t)}{f(t)},
\end{align}
and
\begin{align}\label{ineq7}
\int_{S_0(t)}\int_{0}^{t}|\hat{F}\cdot\bar{\hat{\rho}}|+|\hat{G}\cdot\bar{\hat{u}}|  dsd\xi
&=\int_{0}^{t}\int_{S_0(t)}|\hat{F}\cdot\bar{\hat{\rho}}|+|\hat{G}\cdot\bar{\hat{u}}| d\xi ds  \\ \notag
&\leq C\sqrt{\frac {f'(t)} {f(t)}} \int_{0}^{t}(\|u\|^2_{L^{2}}+\|\rho\|^2_{L^{2}})D_0(s)^{\frac 1 2}ds  \\ \notag
&\leq C\sqrt{\frac {f'(t)} {f(t)}}(1+t)^{\frac 1 2}.
\end{align}
It follows from \eqref{ineq4}-\eqref{ineq7} that
\begin{align}\label{ineq8}
\int_{S_0(t)}|\hat{\rho}|^2+|\hat{u}|^2 d\xi\leq C\ln^{-\frac 1 2}(e+t).
\end{align}
According to \eqref{ineq2} and \eqref{ineq8}, we deduce that
\begin{align}\label{ineq9}
\frac d {dt} [f(t)E_0(t)]\leq Cf'(t)\ln^{-\frac 1 2}(e+t)+2f'(t)\|\nabla^2 \rho\|^2_{L^2},
\end{align}
which implies that
\begin{align*}
f(t)E_0(t)\leq C+C\int_{0}^{t}f'(s)\ln^{-\frac 1 2}(e+s)ds+\int_{0}^{t}f'(s)\|\nabla^2 \rho\|^2_{L^2}ds \leq C\ln^{\frac 5 2}(e+t).
\end{align*}
Then we obtain the initial time decay rate:
\begin{align}\label{ineq10}
E_0(t)\leq C\ln^{-\frac 1 2}(e+t).
\end{align}
We improve the $L^2$ decay rate in \eqref{ineq10} by using the bootstrap argument.
According to \eqref{ineq7} and \eqref{ineq10}, we obtain
\begin{align}\label{ineq11}
\int_{S_0(t)}\int_{0}^{t}|\hat{F}\cdot\bar{\hat{\rho}}|+|\hat{G}\cdot\bar{\hat{u}}|  dsd\xi &\leq C\sqrt{\frac {f'(t)} {f(t)}} \int_{0}^{t}(\|u\|^2_{L^{2}}+\|\rho\|^2_{L^{2}})D_0(s)^{\frac 1 2}ds \\ \notag
&\leq C\sqrt{\frac {f'(t)} {f(t)}}(1+t)^{\frac 1 2}\ln^{-\frac 1 2}(e+t),
\end{align}
where we have used the fact that
\begin{align*}
\lim_{t\rightarrow\infty}\frac {\int_{0}^{t}\ln^{-1}(e+s)ds} {(1+t)\ln^{-1}(e+t)}=\lim_{t\rightarrow\infty}\frac {\ln^{-1}(e+t)} {\ln^{-1}(1+t)-\ln^{-2}(e+t)}  =1.
\end{align*}
Then the proof of \eqref{ineq8} implies that
\begin{align}\label{ineq12}
\int_{S_0(t)}|\hat{\rho}|^2+|\hat{u}|^2 d\xi\leq C\ln^{-1}(e+t).
\end{align}
According to \eqref{ineq2} and \eqref{ineq12}, we obtain
\begin{align*}
E_0\leq C\ln^{-1}(e+t).
\end{align*}
By virtue of the bootstrap argument, for any $l\in N^{+}$, we can deduce that
\begin{align}\label{ineq13}
E_0\leq C\ln^{-l}(e+t).
\end{align}
Taking $\sigma=1$ in \eqref{ineq0}, we have
\begin{align}\label{ineq14}
\frac d {dt} E_1+D_1\leq 0,
\end{align}
which implies that
\begin{align}\label{ineq15}
&\frac d {dt} [f(t)E_1]+C_2 f'(t)(\frac 1 2\|\nabla u\|^2_{H^{1}}+\eta\gamma\|\nabla \rho\|^2_{L^{2}})
+f(t)\|\nabla \tau\|^2_{H^{1}(\mathcal{L}^{2})}  \\ \notag
&\leq Cf'(t)\int_{S_0(t)}|\xi|^2(|\hat{u}(\xi)|^2+|\hat{\rho}(\xi)|^2) d\xi+2f'(t)\|\nabla^2 \rho\|^2_{L^{2}}.
\end{align}
According to \eqref{ineq13}, we get
\begin{align}\label{ineq16}
f'(t)\int_{S_0(t)}|\xi|^2(|\hat{u}(\xi)|^2+|\hat{\rho}(\xi)|^2) d\xi\leq C (1+t)^{-2}\ln^{-l+1}(e+t).
\end{align}
It follows from \eqref{ineq13}-\eqref{ineq16} that
\begin{align*}
(1+t)\ln^{l+1}(e+t)E_1&\leq C+C\ln(e+t)+C\int_{0}^{t}\ln^{l}(e+s)\|\nabla^2 \rho\|^2_{L^{2}}ds+C\int_{0}^{t}\ln^{l+1}(e+s)E_1ds \\
&\leq C\ln(e+t)+C\int_{0}^{t}\ln^{l+1}(e+s)D_0 ds  \\
&\leq C\ln(e+t)+C\int_{0}^{t}(1+s)^{-1}\ln^{l}(e+s)E_0 ds \\
&\leq C\ln(e+t),
\end{align*}
which implies that
\begin{align}\label{ineq17}
E_1\leq C (1+t)^{-1}\ln^{-l}(e+t).
\end{align}
We thus complete the proof of Proposition \ref{prop1}.
\end{proof}

By virtue of the time weighted energy estimate and the improved Fourier splitting method, one can not obtain the optimal decay rate. However, we can obtain a weak result as follow.
\begin{prop}\label{prop2}
Under the condition in Proposition \ref{prop1}, then there exists a constant $C$ such that
\begin{align}\label{decay2}
E_0(t)\leq C(1+t)^{-\frac 1 2},
\end{align}
and
\begin{align}\label{decay3}
E_1(t)\leq C(1+t)^{-\frac 3 2}.
\end{align}
\end{prop}
\begin{proof}
Define $S(t)=\{\xi:|\xi|^2\leq C_2(1+t)^{-1}\}$ where the constant $C_2$ will be chosen later on. Using Schonbek's strategy, we split the phase space into two time-dependent domain:
$$\|\nabla u\|^2_{H^2}=\int_{S(t)}(1+|\xi|^{4})|\xi|^2|\hat{u}(\xi)|^2 d\xi+\int_{S(t)^c}(1+|\xi|^{4})|\xi|^2|\hat{u}(\xi)|^2 d\xi.$$
One can verify that
$$\frac {C_2} {1+t} \int_{S(t)^c}(1+|\xi|^{4})|\hat{u}(\xi)|^2 d\xi\leq\|\nabla u\|^2_{H^2},$$
and
$$\frac {C_2} {1+t} \int_{S(t)^c}(1+|\xi|^{2})|\hat{\rho}(\xi)|^2 d\xi\leq\|\nabla \rho\|^2_{H^{1}}.$$
By \eqref{ineq1}, we have
\begin{align}\label{ineq18}
\frac d {dt} E_0(t)+\frac { C_2} {2(1+t)}\| u\|^2_{H^2}+\frac {\eta\gamma C_2} {1+t}\|\rho\|^2_{H^{1}}
+\|\tau\|^2_{H^{2}}\leq \frac {CC_2} {1+t}\int_{S(t)}|\hat{u}(\xi)|^2+|\hat{\rho}(\xi)|^2 d\xi.
\end{align}
Integrating \eqref{ineq3} over $S(t)$ with $\xi$, then we obtain
\begin{align}\label{ineq19}
\int_{S(t)}|\hat{\rho}|^2+|\hat{u}|^2+|\hat{\tau}|^2d\xi
&\leq C\int_{S(t)} |\hat{\rho}_0|^2+|\hat{u}_0|^2+|\hat{\tau}_0|^2d\xi  \\ \notag
&+C\int_{S(t)}\int_{0}^{t}|\hat{F}\cdot\bar{\hat{\rho}}|+|\hat{G}\cdot\bar{\hat{u}}|+|\hat{H}|^2dsd\xi.
\end{align}
Using the fact that $(\rho_0,u_0,\tau_0)\in H^2\cap\dot{B}^{-1}_{2,\infty}$ and applying Proposition \ref{pro0}, we deduce that
\begin{align}
\int_{S(t)}|\hat{\rho}_0|^2+|\hat{u}_0|^2+|\hat{\tau}_0|^2 d\xi
&\leq\sum_{j\leq \log_2[\frac {4} {3}C_2^{\frac 1 2 }(1+t)^{-\frac 1 2}]}\int_{\mathbb{R}^{2}} 2\varphi^2(2^{-j}\xi)(|\hat{\rho}_0|^2+|\hat{u}_0|^2+|\hat{\tau}_0|^2)d\xi \\ \notag
&\leq\sum_{j\leq \log_2[\frac {4} {3}C_2^{\frac 1 2 }(1+t)^{-\frac 1 2}]}2(\|\dot{\Delta}_j u_0\|^2_{L^2}+\|\dot{\Delta}_j \rho_0\|^2_{L^2}+\|\dot{\Delta}_j \tau_0\|^2_{L^2}) \\ \notag
&\leq C\|(\rho_0,u_0,\tau_0)\|^2_{\dot{B}^{-1}_{2,\infty}}\sum_{j\leq \log_2[\frac {4} {3}C_2^{\frac 1 2 }(1+t)^{-\frac 1 2}]}2^{2j} \\ \notag
&\leq C(1+t)^{-1}\|(\rho_0,u_0,\tau_0)\|^2_{\dot{B}^{-1}_{2,\infty}}.
\end{align}
Thanks to Minkowski's inequality, we have
\begin{align}\label{ineq20}
&\int_{S(t)}\int_{0}^{t}|\hat{F}\cdot\bar{\hat{\rho}}|+|\hat{G}\cdot\bar{\hat{u}}|dsd\xi
\leq C(\int_{S(t)}d\xi)^{\frac 1 2} \int_{0}^{t}\|\hat{F}\cdot\bar{\hat{\rho}}\|_{L^{2}}+\|\hat{G}\cdot\bar{\hat{u}}\|_{L^{2}}ds \\ \notag
&\leq C(1+t)^{-\frac 1 2} \int_{0}^{t}(\|u\|^2_{L^{2}}+\|\rho\|^2_{L^{2}})(\|\nabla u\|_{H^{1}}+\|\nabla\rho\|_{L^{2}}+\|\nabla \tau\|_{L^{2}})ds.
\end{align}
Using \eqref{ineq0}, then we obtain
\begin{align}\label{ineq21}
\int_{S(t)}\int_{0}^{t}|\hat{H}|^2 dsd\xi &\leq C(1+t)^{-1} \int_{0}^{t}\|u\|^2_{L^{2}}\|\nabla \tau\|^2_{L^{2}}+\|\nabla u\|^2_{L^{2}}\|\tau\|^2_{L^{2}}ds \\ \notag
&\leq C(1+t)^{-1}.
\end{align}
It follows from \eqref{ineq19}-\eqref{ineq21} that
\begin{align}\label{ineq22}
&\int_{S(t)}|\hat{\rho}(t,\xi)|^2+|\hat{u}(t,\xi)|^2 d\xi\leq  C(1+t)^{-1}  \\ \notag
&+C(1+t)^{-\frac 1 2} \int_{0}^{t}(\|u\|^2_{L^{2}}+\|\rho\|^2_{L^{2}})(\|\nabla u\|_{H^{1}}+\|\nabla\rho\|_{L^{2}}+\|\nabla\tau\|_{L^{2}})ds.
\end{align}
According to \eqref{ineq18} and \eqref{ineq22}, we obtain
\begin{align*}
&\frac d {dt} E_0(t)+\frac {C_2} {2(1+t)}\| u\|^2_{H^2}+\frac {\eta\gamma C_2} {1+t}\|\rho\|^2_{H^{1}}
+\|\tau\|^2_{H^{2}}  \\ \notag
&\leq \frac {CC_2} {1+t}[(1+t)^{-1}+(1+t)^{-\frac 1 2} \int_{0}^{t}(\|u\|^2_{L^{2}}+\|\rho\|^2_{L^{2}})(\|\nabla u\|_{H^{1}}+\|\nabla\rho\|_{L^{2}}+\|\nabla\tau\|_{L^{2}})ds],
\end{align*}
which implies that
\begin{align}\label{ineq23}
(1+t)^{\frac 3 2}E_0(t)&\leq C\int_{0}^{t}\|\nabla^2 \rho\|^2_{L^{2}}(1+s)^{\frac 1 2}ds+C(1+t)^{\frac 1 2}  \\ \notag
&+C(1+t)\int_{0}^{t}(\|u\|^2_{L^{2}}+\|\rho\|^2_{L^{2}})(\|\nabla u\|_{H^{1}}+\|\nabla\rho\|_{L^{2}}+\|\nabla\tau\|_{L^{2}})ds\\ \notag
&\leq C(1+t)^{\frac 1 2}+C(1+t)\int_{0}^{t}(\|u\|^2_{L^{2}}+\|\rho\|^2_{L^{2}})(\|\nabla u\|_{H^{1}}+\|\nabla\rho\|_{L^{2}}+\|\nabla\tau\|_{L^{2}})ds.
\end{align}
Define $N(t)=\sup_{0\leq s\leq t}(1+s)^{\frac 1 2}E_0(s)$. According to \eqref{ineq23}, we get
\begin{align}\label{ineq24}
N(t)&\leq C+C\int_{0}^{t}(\|u\|^2_{L^{2}}+\|\rho\|^2_{L^{2}})(\|\nabla u\|_{H^{1}}+\|\nabla\rho\|_{L^{2}}+\|\nabla\tau\|_{L^{2}})ds  \\ \notag
&\leq C+C\int_{0}^{t}N(s)(1+s)^{-\frac 1 2}(\|\nabla u\|_{H^{1}}+\|\nabla\rho\|_{L^{2}}+\|\nabla\tau\|_{L^{2}})ds.
\end{align}
Applying Gronwall's inequality, Proposition \ref{prop1}, we obtain $N(t)\leq C$,
which implies that
\begin{align}\label{ineq25}
E_0\leq C(1+t)^{-\frac 1 2}.
\end{align}
By \eqref{ineq14}, we can deduce that
\begin{align}\label{ineq26}
&\frac d {dt} E_1+\frac { C_2} {1+t}(\frac 1 2\|\nabla u\|^2_{H^{1}}+\eta\gamma\|\nabla \rho\|^2_{L^{2}})
+\|\nabla \tau\|^2_{H^{1}}\leq \frac {CC_2} {1+t}\int_{S(t)}|\xi|^2(|\hat{u}(\xi)|^2+|\hat{\rho}(\xi)|^2) d\xi.
\end{align}
According to \eqref{ineq25}, we have
\begin{align*}
\frac {CC_2} {1+t}\int_{S(t)}|\xi|^2(|\hat{u}(\xi)|^2+|\hat{\rho}(\xi)|^2) d\xi\leq C{C_2}^2 (1+t)^{-2}(\|\rho\|^2_{L^2}+\|u\|^2_{L^2})\leq C (1+t)^{-\frac 5 2}.
\end{align*}
This together with \eqref{ineq1}, \eqref{ineq25} and \eqref{ineq26} ensure that
\begin{align}\label{ineq27}
(1+t)^{\frac 5 2}E_1&\leq C(1+t)+C\int_{0}^{t}\|\nabla^2 \rho\|^2_{L^{2}}(1+s)^{\frac 3 2}ds  \\ \notag
&\leq C(1+t)+C\int_{0}^{t}E_0(s)(1+s)^{\frac 1 2}ds   \\ \notag
&\leq C(1+t),
\end{align}
which implies that
\begin{align}\label{ineq28}
E_1\leq C(1+t)^{-\frac 3 2}.
\end{align}
Therefore, we complete the proof of Proposition \ref{prop2}.
\end{proof}

\begin{rema}
The proposition \ref{prop2} indicates that
$$\|\rho\|_{L^2}+\|u\|_{L^2}\leq C(1+t)^{-\frac 1 4}.$$
Combining with the CNS system, one can see that this is not the optimal time decay.
\end{rema}

By Proposition \ref{prop2}, we can prove that the solution of \eqref{eq1} belongs to
some Besov space with negative index \cite{Tong2017The}, which can be used to improve time decay rate.
\begin{lemm}\label{Lemma1}
Let $0<\alpha,\sigma\leq 1$ and $\sigma<2\alpha$. Assume that $(\rho_0,u_0,\tau_0)$ satisfy the condition in Proposition \ref{prop1}. If
\begin{align}\label{de0}
E_0(t)\leq C(1+t)^{-\alpha},~~~~E_1(t)\leq C(1+t)^{-\alpha-1},
\end{align}
then we have
\begin{align}\label{ineq29}
(\rho,u,\tau)\in L^{\infty}(0,\infty;\dot{B}^{-\sigma}_{2,\infty}).
\end{align}
\end{lemm}
\begin{proof}
Applying $\dot{\Delta}_j$ to \eqref{eq1}, we obtain
\begin{align}\label{eq4}
\left\{
\begin{array}{ll}
\dot{\Delta}_j\rho_t+div~\dot{\Delta}_j u=\dot{\Delta}_j F,  \\[1ex]
\dot{\Delta}_j u_t-\frac 1 2 (\Delta+\nabla div)\dot{\Delta}_j u+\gamma\nabla\dot{\Delta}_j \rho-div\dot{\Delta}_j\tau=\dot{\Delta}_j G,  \\[1ex]
\dot{\Delta}_j \tau_t+\dot{\Delta}_j \tau-\dot{\Delta}_j D(u)=\dot{\Delta}_j H. \\[1ex]
\end{array}
\right.
\end{align}
By virtue of the standard energy estimate for \eqref{eq4}, we can deduce that
\begin{align}\label{ineq30}
&\frac d {dt}(\gamma\|\dot{\Delta}_j \rho\|^2_{L^2}+\|\dot{\Delta}_j u\|^2_{L^2}+\|\dot{\Delta}_j \tau\|^2_{L^2})+\|\nabla\dot{\Delta}_j u\|^2_{L^2}+\|div \dot{\Delta}_j u\|^2_{L^2}+2\|\dot{\Delta}_j \tau\|^2_{L^2}   \\ \notag
&=2\int_{\mathbb{R}^{2}} \gamma\dot{\Delta}_j F\dot{\Delta}_j \rho dx+2\int_{\mathbb{R}^{2}} \dot{\Delta}_j G\dot{\Delta}_j u dx+2\int_{\mathbb{R}^{2}} \dot{\Delta}_j H\dot{\Delta}_j \tau dx   \\ \notag
&\leq C(\|\dot{\Delta}_j F\|_{L^2}\|\dot{\Delta}_j \rho\|_{L^2}+\|\dot{\Delta}_j G\|_{L^2}\|\dot{\Delta}_j u\|_{L^2}+\|\dot{\Delta}_j H\|_{L^2}\|\dot{\Delta}_j \tau\|_{L^2}).
\end{align}
Multiplying both sides of \eqref{ineq30} by $2^{-2j\sigma}$ and taking $l^\infty$ norm, we obtain
\begin{align}\label{ineq31}
&\frac d {dt}(\gamma\|\rho\|^2_{\dot{B}^{-\sigma}_{2,\infty}}+\|u\|^2_{\dot{B}^{-\sigma}_{2,\infty}}+\|\tau\|^2_{\dot{B}^{-\sigma}_{2,\infty}})  \\ \notag
&\leq C(\|F\|_{\dot{B}^{-\sigma}_{2,\infty}}\|\rho\|_{\dot{B}^{-\sigma}_{2,\infty}}+\|G\|_{\dot{B}^{-\sigma}_{2,\infty}}\|u\|_{\dot{B}^{-\sigma}_{2,\infty}}
+\|H\|_{\dot{B}^{-\sigma}_{2,\infty}}\|\tau\|_{\dot{B}^{-\sigma}_{2,\infty}}).
\end{align}
Define $M(t)=\sum_{0\leq s\leq t} \|\rho\|_{\dot{B}^{-\sigma}_{2,\infty}}+\|u\|_{\dot{B}^{-\sigma}_{2,\infty}}+\|\tau\|_{\dot{B}^{-\sigma}_{2,\infty}}$. According to \eqref{ineq31}, we deduce that
\begin{align}\label{ineq32}
M^2(t)&\leq CM^2(0)+CM(t)\int_0^{t}\|F\|_{\dot{B}^{-\sigma}_{2,\infty}}+\|G\|_{\dot{B}^{-\sigma}_{2,\infty}}+\|H\|_{\dot{B}^{-\sigma}_{2,\infty}}ds  .
\end{align}
Using \eqref{de0} and Lemmas \ref{Lemma}, \ref{Lemma0}, we obtain
\begin{align}\label{ineq33}
\int_0^{t}\|(F,G,H)\|_{\dot{B}^{-\sigma}_{2,\infty}}ds
&\leq C\int_0^{t}\|(F,G,H)\|_{L^{\frac 2 {\sigma+1}}}ds  \\ \notag
&\leq C\int_0^{t}(\|\rho\|_{L^{\frac 2 {\sigma}}}+\|u\|_{L^{\frac 2 {\sigma}}})(\|\nabla\rho\|_{L^2}+\|\nabla u\|_{H^1}+\|\nabla\tau\|_{L^2})ds   \\ \notag
&\leq C\int_0^{t}\|(\rho,u)\|^\sigma_{L^{2}}\|\nabla(\rho,u)\|^{1-\sigma}_{L^{2}}(\|\nabla\rho\|_{L^2}+\|\nabla u\|_{H^1}+\|\nabla\tau\|_{L^2})ds   \\ \notag
&\leq C\int_0^{t}(1+s)^{-(1+\alpha-\frac \sigma 2)}ds\leq C.
\end{align}
According to \eqref{ineq32} and \eqref{ineq33}, we have $M^2(t)\leq CM^2(0)+CM(t)$. By virtue of interpolation theory, we can deduce that $(\rho_0,u_0,\tau_0)\in \dot{B}^{-\sigma}_{2,\infty}$ with $0<\sigma\leq1$. Then $M^2(0)\leq C$ implies that $M(t)\leq C$.
\end{proof}

\begin{prop}\label{prop3}
Let $0<\beta,\sigma\leq 1$ and $\frac {1} {2} \leq\alpha$. Assume that $(\rho_0,u_0,\tau_0)$ satisfy the condition in Proposition \ref{prop1}.
For any $t\in [0,+\infty)$, if
\begin{align}\label{de1}
E_0(t)\leq C(1+t)^{-\alpha},~~~~E_1(t)\leq C(1+t)^{-\alpha-1},
\end{align}
and
\begin{align}\label{ineq34}
(\rho,u,\tau)\in L^{\infty}(0,\infty;\dot{B}^{-\sigma}_{2,\infty}),
\end{align}
then there exists a constant $C$ such that
\begin{align}\label{de2}
E_0(t)\leq C(1+t)^{-\beta}~~~~and~~~~E_1(t)\leq C(1+t)^{-\beta-1}
\end{align}
where $\beta<\frac {\sigma+1} {2}$ for $\alpha=\frac {1} {2}$ and $\beta=\frac {\sigma+1} {2}$ for $\alpha>\frac {1} {2}$.
\end{prop}
\begin{proof}
According to the proof of Proposition \ref{prop2}, we obtain
\begin{align}\label{ineq35}
&\frac d {dt} E_0(t)+\frac  {C_2} {2(1+t)}\| u\|^2_{H^2}+\frac {\eta\gamma C_2} {1+t}\|\rho\|^2_{H^{1}}
+\|\tau\|^2_{H^{2}}  \\ \notag
&\leq \frac {CC_2} {1+t}((1+t)^{-1}+\int_{S(t)}\int_{0}^{t}|\hat{F}\cdot\bar{\hat{\rho}}|+|\hat{G}\cdot\bar{\hat{u}}|dsd\xi).
\end{align}
By virtue of \eqref{de1} and \eqref{ineq34}, we deduce that
\begin{align}\label{ineq36}
\int_{S(t)}\int_{0}^{t}|\hat{F}\cdot\bar{\hat{\rho}}|+|\hat{G}\cdot\bar{\hat{u}}|dsd\xi
&\leq C\int_{0}^{t}(\|F\|_{L^{1}}\int_{S(t)}|\hat{\rho}|d\xi+\|G\|_{L^{1}}\int_{S(t)}|\hat{u}|d\xi)ds \\ \notag
&\leq C(1+t)^{-\frac 1 2}\int_{0}^{t}(\|F\|_{L^{1}}+\|G\|_{L^{1}})(\int_{S(t)}|\hat{\rho}|^2+|\hat{u}|^2d\xi)^{\frac 1 2} ds  \\ \notag
&\leq C(1+t)^{-\frac {\sigma+1} {2}}M(t)\int_{0}^{t}\|F\|_{L^{1}}+\|G\|_{L^{1}}ds  \\ \notag
&\leq C(1+t)^{-\frac {\sigma+1} {2}}\int_{0}^{t}(1+s)^{-(\alpha+\frac 1 2)}ds  \\ \notag
&\leq C(1+t)^{-\beta}.
\end{align}
Using \eqref{ineq35} and \eqref{ineq36}, we obtain
\begin{align}\label{ineq37}
E_0(t)\leq C(1+t)^{-\beta}.
\end{align}
Recall that
\begin{align*}
\frac d {dt} E_1+\frac {C_2} {1+t}(\frac 1 2\|\nabla u\|^2_{H^{1}}+\eta\gamma\|\nabla \rho\|^2_{L^{2}})
+\|\nabla \tau\|^2_{H^{1}} \leq \frac {CC_2} {1+t}\int_{S(t)}|\xi|^2(|\hat{u}(\xi)|^2+|\hat{\rho}(\xi)|^2) d\xi.
\end{align*}
According to \eqref{ineq37}, we have
\begin{align*}
\frac {CC_2} {1+t}\int_{S(t)}|\xi|^2(|\hat{u}(\xi)|^2+|\hat{\rho}(\xi)|^2) d\xi\leq C{C_2}^2 (1+t)^{-2}(\|\rho\|^2_{L^2}+\|u\|^2_{L^2})\leq C (1+t)^{-2-\beta}.
\end{align*}
Then the proof of \eqref{ineq28} implies that $E_1\leq C(1+t)^{-1-\beta}$.
We thus finish the proof of Proposition \ref{prop3}.
\end{proof}

{\bf The proof of Theorem \ref{th2}:}  \\
We now improve the decay rate in Proposition \ref{prop2}. According to Proposition \ref{prop2} and Lemma \ref{Lemma1} with $\sigma=\alpha=\frac 1 2$, we obtain
\begin{align*}
(\rho,u,\tau)\in L^{\infty}(0,\infty;\dot{B}^{-\frac 1 2}_{2,\infty}).
\end{align*}
Taking advantage of Proposition \ref{prop3} with $\alpha=\sigma=\frac 1 2$ and $\beta=\frac 5 8$, we deduce that
\begin{align*}
E_0(t)\leq C(1+t)^{-\frac 5 8}~~~~and~~~~E_1(t)\leq C(1+t)^{-\frac 5 8-1}.
\end{align*}
Taking $\sigma=1$ and $\alpha=\frac 5 8$ in Lemma \ref{Lemma1}, we obtain
\begin{align*}
(\rho,u,\tau)\in L^{\infty}(0,\infty;\dot{B}^{-1}_{2,\infty}).
\end{align*}
Using Propositions \ref{prop3} again with $\alpha=\frac 5 8$ and $\sigma=\beta=1$, we verify that
\begin{align*}
E_0(t)\leq C(1+t)^{-1}~~~~and~~~~E_1(t)\leq C(1+t)^{-2}.
\end{align*}
To get the faster decay rate for $\tau$ in $L^2$, we need the following standard energy estimation for $\eqref{eq1}_3$:
\begin{align*}
\frac {1} {2}\frac {d} {dt} \|\tau\|^2_{L^2}+\|\tau\|^2_{L^2}\leq \int_{\mathbb{R}^{2}}D(u):\tau dx+\|\nabla u\|_{L^2}\|\tau\|^2_{L^4}.
\end{align*}
Using Lemma \ref{Lemma0}, we deduce that
\begin{align*}
\frac d {dt} \|\tau\|^2_{L^2}+\|\tau\|^2_{L^2}\leq C\|\nabla u\|^2_{L^2}(1+\|\nabla \tau\|^2_{L^2}),
\end{align*}
which implies that
\begin{align*}
\|\tau\|^2_{L^2}
&\leq \|\tau_0\|^2_{L^2}e^{-t}+C\int_{0}^{t}e^{-(t-s)}\|\nabla u\|^2_{L^2}(1+\|\nabla \tau\|^2_{L^2})ds  \\
&\leq C(e^{-t}+\int_{0}^{t}e^{-(t-s)}(1+s)^{-2}ds)  \\
&\leq C(1+t)^{-2}.
\end{align*}
We thus complete the proof of Theorem \ref{th2}.
\hfill$\Box$

\section{Optimal time decay rate for the compressible Hall-MHD equations}
This section is devoted to investigating the large time behaviour for the 2-D compressible Hall-MHD equations.
We introduce the energy and energy dissipation functionals for $(\rho,u,B)$ as follows:
$$\bar{E}_\sigma=\|\Lambda^\sigma(\rho,u,B)\|^2_{H^{2-\sigma}}+2\eta\langle \Lambda^\sigma u, \nabla \Lambda^\sigma\rho \rangle_{H^{1-\sigma}},$$
and
$$\bar{D}_\sigma=\eta\gamma\|\nabla\Lambda^\sigma\rho\|^2_{H^{1-\sigma}}+\|\nabla \Lambda^\sigma u\|^2_{H^{2-\sigma}}+\|div~\Lambda^\sigma u\|^2_{H^{2-\sigma}}+\|\nabla\Lambda^\sigma B\|^2_{H^{2-\sigma}},$$
where $\sigma=0$ or 1 and $\eta$ is a sufficiently small constant.

By virtue of the energy method and interpolation theory in \cite{2017MHDGlobal}, one can easily deduce that the global existence of strong solutions for \eqref{M1}. Thus we omit the proof here and present the following Theorem.
\begin{theo}\label{th3}
Let $d=2$. Let $(\rho,u,B)$ be a strong solution of \eqref{M1} with initial data $(\rho_0,u_0,B_0)\in H^2$. Then, there exists some sufficiently small constant $\epsilon_0$ such that if
\begin{align*}
\bar{E}_0(0)\leq \epsilon_0,
\end{align*}
then \eqref{eq1} admits a unique global strong solution $(\rho,u,B)$. Moreover, for $\sigma=0$ or 1 and any $t>0$, we have
\begin{align}\label{Mineq1}
\frac {d} {dt} \bar{E}_\sigma(t)+\bar{D}_\sigma(t)\leq 0.
\end{align}
\end{theo}
By taking Fourier transform in \eqref{M1} and using the Fourier splitting method, we obtain the initial $L^2$
decay rate in following Proposition.
\begin{prop}\label{prop7}
Let $d=2$. Under the condition in Theorem \ref{th4},
then for any $l\in N^{+}$, there exists a constant $C$ such that
\begin{align*}
\bar{E}_0(t)\leq C\ln^{-l}(e+t),~~~~\bar{E}_1(t)\leq C(1+t)^{-1}\ln^{-l}(e+t).
\end{align*}
\end{prop}
\begin{proof}
Referring to the proof of Propositions \ref{prop1}, one can easily verify that Proposition \ref{prop7} is true. More details can be found in the following Proposition \ref{prop8}.
\end{proof}

By virtue of the improved Fourier splitting method, one can not obtain the optimal decay rate. However, we can obtain a weak result as follow.
\begin{prop}\label{prop8}
Under the condition in Theorem \ref{th4}, there exists a constant $C$ such that
\begin{align*}
\bar{E}_0(t)\leq C(1+t)^{-\frac 1 2},~~~~\bar{E}_1(t)\leq C(1+t)^{-\frac 3 2}.
\end{align*}
\end{prop}
\begin{proof}
Define $S(t)=\{\xi:|\xi|^2\leq C_2(1+t)^{-1}\}$ where the constant $C_2$ will be chosen later on.
Taking $\sigma=0$ in \eqref{Mineq1}, we can easily deduce that
\begin{align}\label{Mineq2}
\frac d {dt} \bar{E}_0(t)+\frac {C_2} {1+t}\|(u,B)\|^2_{H^2}+\frac {\eta\gamma C_2} {1+t}\|\rho\|^2_{H^{1}}
\leq \frac {CC_2} {1+t}\int_{S(t)}|\hat{u}(\xi)|^2+|\hat{\rho}(\xi)|^2+|\hat{B}(\xi)|^2 d\xi.
\end{align}
The $L^2$ estimate to the low frequency part of $(\rho,u,B)$ play a key role in studying time decay rates. Applying Fourier transform to \eqref{M1}, we obtain
\begin{align}\label{M2}
\left\{
\begin{array}{ll}
\hat{\rho}_t+i\xi_{k} \hat{u}^k=\hat{F}_1,  \\[1ex]
\hat{u}^{j}_t+|\xi|^2 \hat{u}^j+\xi_{j} \xi_{k} \hat{u}^k+i\xi_{j} \gamma\hat{\rho}=\hat{G}_1^j,  \\[1ex]
\hat{B}^{k}_t+|\xi|^2\hat{B}^{k}=\hat{H}_1^{k}.\\[1ex]
\end{array}
\right.
\end{align}
One can verify that
$$\mathcal{R}e[i\xi\cdot\hat{u}\bar{\hat{\rho}}]+\mathcal{R}e[\hat{\rho}i\xi\cdot\bar{\hat{u}}]=0.$$
According to \eqref{M2}, then we have
\begin{align}\label{Mineq3}
\int_{S(t)}|\hat{\rho}|^2+|\hat{u}|^2+|\hat{B}|^2d\xi &\leq C\int_{S(t)} |\hat{\rho}_0|^2+|\hat{u}_0|^2+|\hat{B}_0|^2d\xi \\ \notag
&+C\int_{S(t)}\int_{0}^{t}|\hat{F}_1\cdot\bar{\hat{\rho}}|
+|\hat{G}_1\cdot\bar{\hat{u}}|+|\hat{H}_1\cdot\bar{\hat{B}}| dsd\xi.
\end{align}
By $(\rho_0,u_0,B_0)\in H^2\cap\dot{B}^{-1}_{2,\infty}$ and Proposition \ref{pro0}, we deduce that
\begin{align}\label{Mineq4}
\int_{S(t)}|\hat{\rho}_0|^2+|\hat{u}_0|^2+|\hat{B}_0|^2 d\xi
&\leq\sum_{j\leq \log_2[\frac {4} {3}C_2^{\frac 1 2 }(1+t)^{-\frac 1 2}]}\int_{\mathbb{R}^{2}} 2\varphi^2(2^{-j}\xi)(|\hat{\rho}_0|^2+|\hat{u}_0|^2+|\hat{B}_0|^2)d\xi \\ \notag
&\leq\sum_{j\leq \log_2[\frac {4} {3}C_2^{\frac 1 2 }(1+t)^{-\frac 1 2}]}2(\|\dot{\Delta}_j u_0\|^2_{L^2}+\|\dot{\Delta}_j \rho_0\|^2_{L^2}+\|\dot{\Delta}_j \tau_0\|^2_{L^2}) \\ \notag
&\leq \sum_{j\leq \log_2[\frac {4} {3}C_2^{\frac 1 2 }(1+t)^{-\frac 1 2}]}C 2^{2j} \\ \notag
&\leq C(1+t)^{-1}.
\end{align}
Thanks to Minkowski's inequality, we have
\begin{align}\label{Mineq5}
&\int_{S(t)}\int_{0}^{t}|\hat{F}_1\cdot\bar{\hat{\rho}}|+|\hat{G}_1\cdot\bar{\hat{u}}|+|\hat{H}_1\cdot\bar{\hat{B}}|dsd\xi  \\ \notag
&\leq C(\int_{S(t)}d\xi)^{\frac 1 2} \int_{0}^{t}\|\hat{F}_1\cdot\bar{\hat{\rho}}\|_{L^{2}}+\|\hat{G}_1\cdot\bar{\hat{u}}\|_{L^{2}}+\|\hat{H}_1\cdot\bar{\hat{B}}\|_{L^{2}}ds \\ \notag
&\leq C(1+t)^{-\frac 1 2} \int_{0}^{t}\|(\rho,u,B)\|^2_{L^{2}}(\|\nabla u\|_{H^{1}}+\|\nabla\rho\|_{L^{2}}+\|\nabla B\|_{H^{1}})ds,
\end{align}
where we used the fact that
\begin{align*}
\|curl[\frac {(curlB)\times B}{\rho+1}]\|_{L^1}
&\leq C(\|B\|_{L^2}\|\nabla^2 B\|_{L^2}+\|\nabla B\|^2_{L^2}+\|B\|_{L^2}\|\nabla B\|_{L^4}\|\nabla \rho\|_{L^4})  \\
&\leq C\|B\|_{L^2}\|\nabla B\|_{H^1}.
\end{align*}
It follows from \eqref{Mineq3}-\eqref{Mineq5} that
\begin{align}\label{Mineq6}
&\int_{S(t)}|\hat{\rho}(t,\xi)|^2+|\hat{u}(t,\xi)|^2+|\hat{B}(t,\xi)|^2 d\xi\leq  C(1+t)^{-1}  \\ \notag
&+C(1+t)^{-\frac 1 2} \int_{0}^{t}\|(\rho,u,B)\|^2_{L^{2}}(\|\nabla u\|_{H^{1}}+\|\nabla\rho\|_{L^{2}}+\|\nabla B\|_{H^{1}})ds.
\end{align}
According to \eqref{Mineq2} and \eqref{Mineq6}, we obtain
\begin{align*}
&\frac d {dt} \bar{E}_0(t)+\frac {C_2} {1+t}\|(u,B)\|^2_{H^2}+\frac {\eta\gamma C_2} {1+t}\|\rho\|^2_{H^{1}}  \\ \notag
&\leq \frac {CC_2} {1+t}[(1+t)^{-1}+(1+t)^{-\frac 1 2} \int_{0}^{t}\|(\rho,u,B)\|^2_{L^{2}}(\|\nabla u\|_{H^{1}}+\|\nabla\rho\|_{L^{2}}+\|\nabla B\|_{H^{1}})ds],
\end{align*}
which implies that
\begin{align}\label{Mineq7}
(1+t)^{\frac 3 2}\bar{E}_0(t)&\leq C\int_{0}^{t}\|\nabla^2 \rho\|^2_{L^{2}}(1+s)^{\frac 1 2}ds+C(1+t)^{\frac 1 2}  \\ \notag
&+C(1+t)\int_{0}^{t}\|(\rho,u,B)\|^2_{L^{2}}(\|\nabla u\|_{H^{1}}+\|\nabla\rho\|_{L^{2}}+\|\nabla B\|_{L^{2}})ds\\ \notag
&\leq C(1+t)^{\frac 1 2}+C(1+t)\int_{0}^{t}\|(\rho,u,B)\|^2_{L^{2}}(\|\nabla u\|_{H^{1}}+\|\nabla\rho\|_{L^{2}}+\|\nabla B\|_{H^{1}})ds.
\end{align}
Define $\bar{N}(t)=\sup_{0\leq s\leq t}(1+s)^{\frac 1 2}\bar{E}_0(s)$. According to \eqref{Mineq7}, we get
\begin{align}\label{Mineq8}
\bar{N}(t)\leq C+C\int_{0}^{t}N(s)(1+s)^{-\frac 1 2}(\|\nabla u\|_{H^{1}}+\|\nabla\rho\|_{L^{2}}+\|\nabla B\|_{H^{1}})ds.
\end{align}
Applying Gronwall's inequality and Proposition \ref{prop7}, we obtain $N(t)\leq C$,
which implies that
\begin{align}\label{Mineq9}
\bar{E}_0\leq C(1+t)^{-\frac 1 2}.
\end{align}
Taking $\sigma=1$ in \eqref{Mineq1}, we have
\begin{align}\label{Mineq10}
\frac d {dt} \bar{E}_1+\bar{D}_1\leq 0,
\end{align}
According to \eqref{Mineq9} and \eqref{Mineq10} we deduce that
\begin{align}\label{Mineq11}
&\frac d {dt} \bar{E}_1+\frac {C_2} {1+t}(\|\nabla (u,B)\|^2_{H^{1}}+\eta\gamma\|\nabla \rho\|^2_{L^{2}})
\leq C(1+t)^{-\frac 5 2}.
\end{align}
By \eqref{Mineq1}, we obtain
\begin{align*}
(1+t)^{\frac 5 2}\bar{E}_1&\leq C(1+t)+C\int_{0}^{t}\|\nabla^2 \rho\|^2_{L^{2}}(1+s)^{\frac 3 2}ds  \\ \notag
&\leq C(1+t)+C\int_{0}^{t}\bar{E}_0(s)(1+s)^{\frac 1 2}ds   \\ \notag
&\leq C(1+t),
\end{align*}
which implies that
\begin{align*}
\bar{E}_1\leq C(1+t)^{-\frac 3 2}.
\end{align*}
Therefore, we complete the proof of Proposition \ref{prop8}.
\end{proof}

By Proposition \ref{prop8}, we can prove that the solution of \eqref{M1} belongs to
some Besov space with negative index.
\begin{lemm}\label{Lemma2}
Let $0<\alpha,\sigma\leq 1$ and $\sigma<2\alpha$. Assume that $(\rho_0,u_0,B_0)$ satisfy the same condition in Theorem \ref{th4}. For any $t\in [0,+\infty)$, if
\begin{align}\label{de3}
\bar{E}_0(t)\leq C(1+t)^{-\alpha},~~~~\bar{E}_1(t)\leq C(1+t)^{-\alpha-1},
\end{align}
then we have
\begin{align}\label{Mineq12}
(\rho,u,B)\in L^{\infty}(0,\infty;\dot{B}^{-\sigma}_{2,\infty}).
\end{align}
\end{lemm}
\begin{proof}
Applying $\dot{\Delta}_j$ to \eqref{M1}, we get
\begin{align}\label{M3}
\left\{
\begin{array}{ll}
\dot{\Delta}_j\rho_t+div~\dot{\Delta}_j u=\dot{\Delta}_j F_1,  \\[1ex]
\dot{\Delta}_j u_t-(\Delta+\nabla div)\dot{\Delta}_j u+\gamma\nabla\dot{\Delta}_j \rho=\dot{\Delta}_j G_1,  \\[1ex]
\dot{\Delta}_j B_t-\Delta\dot{\Delta}_j B=\dot{\Delta}_j H_1. \\[1ex]
\end{array}
\right.
\end{align}
Using the standard estimate in Besov spaces for \eqref{M3}, we obtain
\begin{align}\label{Mineq13}
&\frac d {dt}(\gamma\|\rho\|^2_{\dot{B}^{-\sigma}_{2,\infty}}+\|u\|^2_{\dot{B}^{-\sigma}_{2,\infty}}+\|B\|^2_{\dot{B}^{-\sigma}_{2,\infty}(\mathcal{L}^{2})})  \\ \notag
&\leq C(\|F\|_{\dot{B}^{-\sigma}_{2,\infty}}\|\rho\|_{\dot{B}^{-\sigma}_{2,\infty}}+\|G\|_{\dot{B}^{-\sigma}_{2,\infty}}\|u\|_{\dot{B}^{-\sigma}_{2,\infty}}+\|H\|_{\dot{B}^{-\sigma}_{2,\infty}}\|B\|_{\dot{B}^{-\sigma}_{2,\infty}}).
\end{align}
Define $\bar{M}(t)=\sum_{0\leq s\leq t} \|\rho\|_{\dot{B}^{-\sigma}_{2,\infty}}+\|u\|_{\dot{B}^{-\sigma}_{2,\infty}}+\|B\|_{\dot{B}^{-\sigma}_{2,\infty}}$. According to \eqref{Mineq13}, we deduce that
\begin{align}\label{Mineq14}
\bar{M}^2(t)&\leq C\bar{M}^2(0)+C\bar{M}(t)\int_0^{t}\|(F_1,G_1,H_1)\|_{\dot{B}^{-\sigma}_{2,\infty}}ds  .
\end{align}
Using \eqref{de3} and Lemmas \ref{Lemma}, \ref{Lemma0}, we obtain
\begin{align}\label{Mineq15}
\int_0^{t}\|(F_1,G_1,H_1)\|_{\dot{B}^{-\sigma}_{2,\infty}}ds
&\leq C\int_0^{t}\|(F_1,G_1,H_1)\|_{L^{\frac 2 {\sigma+1}}}ds  \\ \notag
&\leq C\int_0^{t}\|(\rho,u,B)\|_{L^{\frac 2 \sigma}}\|\nabla (\rho,u,B)\|_{H^{1}}+\|\nabla B\|_{L^{2}}\|\nabla (\rho,B)\|_{L^{\frac 2 \sigma}}ds   \\ \notag
&\leq C\int_0^{t}(1+s)^{-(1+\alpha-\frac \sigma 2)}ds\leq C.
\end{align}
Then we have $\bar{M}^2(t)\leq C\bar{M}^2(0)+C\bar{M}(t)$.  By virtue of interpolation theory, we have $\bar{M}^2(0)\leq C$, which implies that $\bar{M}(t)\leq C$.
\end{proof}

{\bf The proof of Theorem \ref{th4}:}  \\
According to the proof of Proposition \ref{prop8}, we obtain
\begin{align}\label{Mineq16}
&\frac d {dt} \bar{E}_0(t)+\frac  {C_2} {1+t}\|(u,B)\|^2_{H^2}+\frac {\eta\gamma C_2} {1+t}\|\rho\|^2_{H^{1}} \\ \notag
&\leq \frac {CC_2} {1+t}((1+t)^{-1}+\int_{S(t)}\int_{0}^{t}|\hat{F}_1\cdot\bar{\hat{\rho}}|+|\hat{G}_1\cdot\bar{\hat{u}}|+|\hat{H}_1\cdot\bar{\hat{B}}|dsd\xi).
\end{align}
By virtue of Propositions \ref{prop8} and Lemma \ref{Lemma2} with $\alpha=\sigma=\frac 1 2$, we have
\begin{align}\label{Mineq17}
\int_{S(t)}\int_{0}^{t}|\hat{F}\cdot\bar{\hat{\rho}}|+|\hat{G}\cdot\bar{\hat{u}}|+|\hat{H}_1\cdot\bar{\hat{B}}|dsd\xi
&\leq C(1+t)^{-\frac 3 4}\bar{M}(t)\int_{0}^{t}\|(F_1,G_1,H_1)\|_{L^{1}}ds  \\ \notag
&\leq C(1+t)^{-\frac 3 4}\int_{0}^{t}(1+s)^{-1}ds  \\ \notag
&\leq C(1+t)^{-\frac 5 8}.
\end{align}
Using \eqref{Mineq10}, \eqref{Mineq16} and \eqref{Mineq17}, we deduce that
\begin{align}\label{Mineq18}
\bar{E}_0(t)\leq C(1+t)^{-\frac 5 8},~~~~\bar{E}_1(t)\leq C(1+t)^{-\frac 5 8-1}
\end{align}
Taking $\alpha=\frac 5 8$ and $\sigma=1$ in \eqref{Mineq15} and Lemma \ref{Lemma2}, we have
\begin{align}\label{Mineq19}
\int_{S(t)}\int_{0}^{t}|\hat{F}_1\cdot\bar{\hat{\rho}}|+|\hat{G}_1\cdot\bar{\hat{u}}|+|\hat{H}_1\cdot\bar{\hat{B}}|dsd\xi
&\leq C(1+t)^{-1}\bar{M}(t)\int_{0}^{t}\|(F_1,G_1,H_1)\|_{L^{1}}ds  \\ \notag
&\leq C(1+t)^{-1}.
\end{align}
According to \eqref{Mineq16} and \eqref{Mineq19}, we obtain
$\bar{E}_0(t)\leq C(1+t)^{-1}$.
By \eqref{Mineq10}, we can deduce that $\bar{E}_1\leq C(1+t)^{-2}$.
We thus complete the proof of Theorem \ref{th4}.
\hfill$\Box$

\smallskip
\noindent\textbf{Acknowledgments} This work was
partially supported by the National Natural Science Foundation of China (No.12171493 and No.11671407), the Macao Science and Technology Development Fund (No. 0091/2018/A3), Guangdong Province of China Special Support Program (No. 8-2015),
the key project of the Natural Science Foundation of Guangdong province (No. 2016A030311004), and National Key R$\&$D Program of China (No. 2021YFA1002100).


\noindent\textbf{Data Availability.}
The data that support the findings of this study are available on citation. The data that support the findings of this study are also available
from the corresponding author upon reasonable request.

\phantomsection
\addcontentsline{toc}{section}{\refname}
\bibliographystyle{abbrv}
\bibliography{OldroydBref}

\end{document}